\documentclass[10pt,reqno]{amsart}
\usepackage{amssymb,mathrsfs,bm,enumerate}
\usepackage{graphicx,xcolor,multicol}
\usepackage{amsfonts,dsfont,mathptmx,amsaddr}
\usepackage{todonotes}

\makeatletter

\newtheorem{theorem}{Theorem}[section]
\newtheorem{lemma}[theorem]{Lemma}
\newtheorem{corollary}[theorem]{Corollary}

\theoremstyle{definition}
\newtheorem{remark}[theorem]{Remark}

\newtheorem*{ass1}{Assumption (A1)}
\newtheorem*{ass2}{Assumption (A2)}

\numberwithin{equation}{section}

\newcommand\E{\mathbb{E}}
\newcommand\R{\mathbb{R}}
\renewcommand\P{\mathbb{P}}

\def\Ric{{\operatorname{Ric}}} 
\def\Cut{{\operatorname{Cut}}}

\def\r{\right} \def\l{\left} \def\e{\operatorname{e}}
\newcommand\newdot{{\kern.8pt\cdot\kern.8pt}}

\def\<{\langle} \def\>{\rangle}  
 \def\1{\mathds{1}}

   \def\th{\theta}
 \def\ka{\kappa}

\def\sC{\mathscr C} \def\sP{\mathscr P}

\begin{document}

\title[Exponential contraction in Wasserstein distance] {Exponential
  contraction in Wasserstein distance \\ on static and evolving
  manifolds}

\author[L.J.~Cheng, A. Thalmaier, S.Q.~Zhang]{Li-Juan
  Cheng\textsuperscript{1,2}, Anton Thalmaier\textsuperscript{1} and
  Shao-Qin Zhang\textsuperscript{3}}

\address{\textsuperscript{\rm 1}Department of Mathematics, Maison du Nombre, University of Luxembourg\\
  4364 Esch-sur-Alzette, Luxembourg}
\address{\textsuperscript{\rm 2}Department of Applied Mathematics, Zhejiang University of Technology\\
  Hangzhou 310023, The People's Republic of China}
\address{\textsuperscript{\rm 3}School of Statistics and Mathematics,
  Central University of Finance and Economics\\ Beijing 100081, The
  People's Republic of China}

\email{lijuan.cheng@uni.lu \text{\rm and} chenglj@zjut.edu.cn}
\email{anton.thalmaier@uni.lu}
\email{zhangsq@cufe.edu.cn}

\keywords{Wasserstein distance, diffusion process, exponential contraction,
  Ricci curvature, Ricci flow}
\subjclass[2010]{60J60, 58J65, 53C44} \date{\today}

\maketitle
{\centering\small Dedicated to the memory of Nicu Boboc\par}

\begin{abstract}
  In this article, exponential contraction in Wasserstein distance for
  heat semigroups of diffusion processes on Riemannian manifolds is
  established under curvature conditions where Ricci curvature is not
  necessarily required to be non-negative.  Compared to the results of
  Wang (2016), we focus on explicit estimates for the exponential
  contraction rate.  Moreover, we show that our results extend to
  manifolds evolving under a geometric flow.  As application, for the
  time-inhomogeneous semigroups, we obtain a gradient estimate with an
  exponential contraction rate under weak curvature conditions, as
  well as uniqueness of the corresponding evolution system of
  measures.

\end{abstract}

\section{Introduction}
Let $M$ be a $d$-dimensional connected Riemannian manifold and
consider the operator $L=\Delta+Z$ where $\Delta$ is the
Laplace-Beltrami operator and $Z$ a $C^1$-vector field on
$M$. We denote by $X_t$ the diffusion process with generator $L$, which is
characterized by the property that
for any test function $f$ on $M$, the relation
$$df(X_t)-Lf(X_t)\,dt=0,\quad\text{modulo differentials of martingals,}$$
holds in the It\^o sense.
Throughout the paper we assume that the $L$-diffusion process is
non-explosive.  This holds true, in particular, when the Bakry-\'Emery
Ricci curvature of $M$ is bounded from below, that is, for some real
constant $K$,
\begin{align}\label{RicZ}
  \Ric^Z(X,X):=\Ric(X,X)-\langle\nabla_X Z,X\rangle \geq K |X|^2,\quad
  X\in T_xM,\ x\in M.
\end{align}
Let $P_t$ be the Markov transition semigroup associated to $X_t$ and
$\mu P_t$ be the law of $X_t$ with initial distribution $\mu$.  It is
well known that there are various functional inequalities on $P_t$ which
all give conditions equivalent to the curvature condition
\eqref{RicZ}, see \cite{BGL14, Wbook14}.

In this article, we investigate $L^q$-Wasserstein contraction
inequalities ($q\geq 1$) for $\mu P_t$.  Denote by $\sP(M)$ the set of
probability measures on $M$.  On $\sP(M)$ the
$L^q$-Wasserstein distance is defined as
$$W_q(\mu_1,\mu_2):=\inf_{\pi\in\sC(\mu_1,\mu_2)}
\left(\iint_{M\times M}\rho(x,y)^q\,d\pi( x, y)\right)^{1/q},
\quad \mu_1,\mu_2\in\sP(M),$$ where
$\rho$ denotes the Riemannian distance on $M$ and $\sC(\mu_1,\mu_2)$
consists of all couplings of $\mu_1$ and~$\mu_2$.  The Wasserstein
distance has various characterizations and plays an important role in
the study of SDEs, partial differential equations, optimal
transportation problem, etc.  For more background, one may consult
\cite{vonRS05,Kuwada13,Wbook14} and the references therein.

The $L^q$-Wasserstein distance $W_q$ on $\sP(M)$ will be used to
quantify the time evolution of $(\mu P_t)_{t\geq 0,\,\mu\in\sP(M)}$.
A typical phenomenon of interest for the system
$(\mu P_t)_{t\geq 0,\,\mu\in\sP(M)}$ is exponential contraction in
the Wasserstein distance, i.e.
\begin{align}\label{intr-Wq}
  W_q(\mu_1P_t,\mu_2 P_t)\leq c\,\e^{-\kappa t}W_q(\mu_1,\mu_2),\quad t\geq 0,\ q\geq 1,
\end{align}
with positive constants $c$ and $\kappa$. We refer the reader to
\cite{Eberle11,Eberle16,LW16,zhang} for work in this direction on the
Euclidean space $M=\R^d$. When $M$ is a Riemannian manifold, for instance, under the
curvature condition
\begin{align}\label{low-cd}
  \Ric^Z(X,X)\geq \ka|X|^2
\end{align}
with $\ka\geq0$, the exponential contraction \eqref{intr-Wq} holds
with $c=1$ and $\ka$ the curvature bound
in~\eqref{low-cd}. Moreover, it is well-known that inequality
\eqref{intr-Wq} with $c=1$ is actually equivalent to the lower
curvature bound~\eqref{low-cd}.  For certain cases, when $\Ric^Z$ is
not bounded from below by zero, Wang \cite{Wang16} showed that the
following inequality holds: for any $q\geq 1$,
\begin{align}\label{W-add-1}
  W_q(\delta_x P_t,\delta_y P_t)\leq c\e^{-{\lambda} t}\left(\rho(x,y)\vee \rho(x,y)^{1/q}\right)
\end{align}
for some constant $c>1$ and $\lambda>0$.

In order to weaken condition \eqref{low-cd} as in \cite{Wang16},
let us first recall the definition of the index:
\begin{align*}
  I^Z(x,y)=I(x,y)+\langle Z,\nabla \rho(\newdot, y)\rangle+\langle Z,\nabla \rho(\newdot,x)\rangle,
  \quad x,y\in M,
\end{align*}
where
\begin{align*}
  I(x,y)=\int_0^{\rho(x,y)}\sum_{i=1}^{d-1}\l\{|\nabla_{\dot{\gamma}}J_i|^2-\langle R(\dot{\gamma},J_i)\dot{\gamma}, J_i \rangle \r\}(\gamma_s)\,d s.
\end{align*}
Here $\rho$ is the distance function, $R$ the Riemann curvature tensor
of $M$, $\gamma\colon [0,\rho(x,y)]\to M$ the minimal geodesic from
$x$ to $y$ with unit speed, $(J_i)_{i=1,\ldots,d-1}$ are Jacobi fields
along $\gamma$ such that
\begin{align*}
  J_i(y)=P_{x,y}J_i(x),\quad i=1,\ldots, d-1,
\end{align*}
for the parallel transport $P_{x,y}\colon T_xM \rightarrow T_yM$
along the geodesic $\gamma$, and
$\{\dot{\gamma}(s), J_i(s)\colon\,1\leq i\leq d-1\}$
($s=0,\ \rho(x,y)$) is an orthonormal basis of the tangent space (at
point $x$ and $y$, respectively). Note that when
$(x,y)\in \Cut(M)$, that is if $x$ is in the cut-locus of $y$,
the minimal geodesic may be not unique. As it is a common convention in
the literature, all conditions on the index $I^Z$ are supposed to hold
outside of $\Cut(M)$. If there exists positive constants $K_1$
and $K_2$ such that
\begin{align}\label{IZ-cond}
  I^Z(x,y)\leq \left((K_1+K_2)\1_{\{\rho(x,y)\leq r_0\}}-K_2\right)\rho(x,y)
\end{align}
and $\Ric^Z$ is bounded below, then \eqref{intr-Wq} holds with $\ka>0$
and $c>1$ for any $q\geq 1$, see \cite{Wang16}. This is the case, for
instance, when $\Ric^Z$ is positive outside a compact set.  It is
crucial that the exponential rate $\lambda$ is independent of $p$.
Due to the equivalence of \eqref{intr-Wq} with $c=1$ and
\eqref{low-cd}, in the negative curvature case it is essential that
$c>1$.

In this paper, we give quantitative estimates of $\ka$ and $c$ by
constructing a suitable auxiliary function. We begin the discussion
with a more general condition (see Assumption~(A1) below) which
includes situation \eqref{IZ-cond}. Actually, we rewrite condition
\eqref{IZ-cond} as follows:
\begin{align*}
  I^Z(x,y)\leq k_1-k_2\rho(x,y),
\end{align*}
for some constants $k_1\geq 0$ and $k_2>0$. Then, for $p>1$,
$t\geq 0$, and $x,y \in M$, we obtain (see Corollary~\ref{cor1} below)
that
\begin{align*}
  W_{p}(\delta_xP_t, \delta_yP_t)\leq \l(1+\frac{2k_1}{k_2}\r)^{(p-1)/p} \exp\l(\frac{k_1^2}{pk_2}-\frac{k_2}{2p\e^{k_1^2/k_2}}t\r)\l(\rho(x,y)\vee\rho(x,y)^{1/p}\r).
\end{align*}
Note that the constant ${k_2}/{(2\e^{k_1^2/k_2})}$ is independent of
$p$.

In the second part of the paper, we extend the results from Riemannian
manifolds to the differentiable manifolds carrying a geometric flow of
complete Riemannian metrics. More precisely, for some
$T_c\in (-\infty,\infty]$, we consider the situation of a
$d$-dimensional differentiable manifold $M$ equipped with a $C^1$
family of complete Riemannian metrics $(g_t)_{t\in (-\infty,T_c)}$.
Let $L_t=\Delta_t+Z_t$, where $\Delta_t$ is the Laplace-Beltrami
operator associated with the metric $g_t$ and $(Z_t)_{t\in [0,T_c)}$
is a $C^{1}$-family of vector fields on $M$.  Assume that the
diffusion process $(X_t)$ generated by $L_t$ is non-explosive before
time $T_c$ (see \cite{ACT08} for detailed construction).  Let
$P_{s,t}$ be the corresponding time-inhomogeneous semigroup.

In \cite{Ch17}, the first author has shown that if
\begin{align*}
  \left(\Ric_t^Z-\frac12\partial_tg_t\right)(X,X)(x)\geq \kappa\,|X|_{t}^2(x)
\end{align*}
for some positive constant $\kappa$, where $\Ric_t^Z$ is defined as in
\eqref{RicZ} for the manifold $(M,g_t)$, then exponential contraction
in $L^p$-Wasserstein distance holds with respect to the
$g_t$-Riemannian distance $\rho_t$.

In this paper, we consider situations where
$\Ric_t^Z-\frac12\partial_tg_t$ is not necessarily bounded below by
zero.  More precisely, assuming that there exists a real-valued
function $k$ such that $\liminf_{r\rightarrow \infty}k(r)>0$ and
\begin{align*}
  \left(\Ric_t^Z-\frac12\partial_tg_t\right)(X,X)(x)\geq k(\rho_t(x))|X|_{t}^2(x),
\end{align*}
we prove that
\begin{align}\label{intr-Wq-2}
  \tilde{W}_{p,s}(\mu_1P_{s,t},\mu_2 P_{s,t})\leq c\e^{-\frac 1p\lambda (t-s)}\tilde{W}_{p,t}(\mu_1,\mu_2),
  \quad t\geq s,\ p\geq 1,
\end{align}
holds for some positive constants $c$ and $\lambda$, where
$$ \tilde{W}_{ p,t}(\mu_1,\mu_2)=\inf_{\pi\in\mathcal{C} (\mu_1,\mu_2)}\l(\int_{M\times M} \rho_t(x,y)^p\vee \rho_t(x,y) \,\pi(d x, d y)\r)^{1/p}.$$
Moreover, in Theorem \ref{main-them2} we give estimates for the
constants $c$ and $\lambda$ and apply these results to estimates of
the semigroup.

Furthermore, we use the $W_{1,t}$-contraction property to prove
uniqueness of the evolution system of measures. It is well known that
invariant measure provide important tools in the study of the long
behavior of diffusion processes. When it comes to time-inhomogeneous
diffusions, the evolution system of measures plays a role similar to
the invariant measure.  In \cite{ChT18}, the first two authors
investigated existence and uniqueness of evolution systems of
measures. In particular, they found that $W_1$-contraction of the
distance helps to prove uniqueness properties (see \cite{ChT18} for
details).  Since now the $W_1$-contraction is established even in
cases when the lower bound of the curvature may be negative, this
allows to improve the result in \cite{ChT18} where a uniform lower
curvature bound had been imposed for each time.  Inspired by this, in
Section 4, we consider uniqueness of the evolution system of measures
under a new relaxed curvature condition which allows a lower bound of
curvature depending on the radial distance (see Theorem
\ref{cor4}). It is surprising that under this new condition, a type of
dimension-free Harnack inequality can be derived which then may be
used to obtain supercontractivity of the semigroup $P_{s,t}$ (see
Theorem \ref{supercontractive-them}).

The paper is structured as follows. In Section 2, we investigate
\eqref{W-add-1} by constructing a suitable coupling $(X,Y)$ and using
a new auxiliary function to measure the distance of $X$ and $Y$. Our
result in this section can be applied to the time-inhomogeneous
diffusion process on manifolds carrying geometric flows in Section
3. Section 4 is devoted to the study of existence of evolution system
of measures under the new kind of curvature condition. Finally,
supercontractivity of the semigroup $P_{s,t}$ with respect to the
evolution system of measure is studied by establishing dimension-free
Harnack inequalities.  \bigskip

\section{Exponential contraction in Wasserstein distance}

We begin this section by specifying our assumptions.

\begin{ass1}
  \it There exist a non-negative continuous function $k_1$ on
  $(0,\infty)$, a positive constant $k_2$ and and a constant $\th\geq 0$ such that
  \begin{align}\label{main-cond}
    I^Z(x,y)\leq k_1(\rho(x,y))-k_2\rho(x,y)^{1+\theta}
  \end{align}
  and such that for some positive
  constants $r_0$ and $k_3$ (with $k_3<k_2$) the following two
  conditions hold:
  \begin{enumerate}[\rm(1)]
  \item $k_1(r)-k_2r^{1+\theta}\leq -k_3r^{1+\theta}$,\quad for
    $r\geq r_0$,
  \item $\int_0^rk_1(v)\, dv<\infty$, \quad for each $r>0$.
  \end{enumerate}
\end{ass1}

\begin{remark}\label{remark}
  Note that if $\Ric^Z(x)\geq k(\rho(x))$ and
$\liminf_{r\rightarrow \infty} k(r)>0$, then there exist constants $k_1$
and $k_2$ such that
\begin{align}\label{main-cond1}
  I^Z(x,y)\leq k_1-k_2\rho(x,y).
\end{align}
In this case, Assumption (A1) is satisfied with $k_1$ a non-negative
constant and $\theta=0$.
\end{remark}

We now state some exponential contraction inequalities for the
Wasserstein distance with explicit estimates of the decay rate.

 \begin{theorem}\label{main-them}
   Suppose that Assumption $(A1)$ holds. Then,
   \begin{enumerate}[\rm(i)]
   \item for $p> 1$, $t\geq 0$, and $x,y \in M$, we have
     \begin{align*}
       W_{p}(\delta_xP_{t},\delta_yP_{t})\leq c_p\e^{-\lambda t/p}(\rho(x,y)\vee\rho(x,y)^{1/p}),
     \end{align*}
     where
     \begin{align*}
       c_p&=(1+r_0)^{(p-1)/p}\exp\l(\frac1{4p}\int_0^{r_0}k_1(r)\,dr+\frac{k_2}{8p}r_0^{2+\theta}\r)\\ \intertext{and}
       \lambda&=k_3r_0^\theta\exp\l(-\frac1{4}\int_0^{r_0}k_1(r)\,dr-\frac{k_2}{8}r_0^{2+\theta}\r);
     \end{align*}
   \item for $t\geq 0$, $\mu_1,\mu_2\in \mathcal{P}(M)$ and $p>1$, we
     have
     \begin{align*}
       \tilde{W}_p(\mu_1 P_t, \mu_2P_t)\leq c_p\e^{-\lambda t/p} \tilde{W}_p(\mu_1, \mu_2),
     \end{align*}
     where
     $$ \tilde{W}_{ p}(\mu_1,\mu_2)=\inf_{\pi\in\mathcal{C}
       (\mu_1,\mu_2)}\l(\int_{M\times M} \rho(x,y)^p\vee \rho(x,y)
     \pi(d x, d y)\r)^{1/p};$$

   \item for $t\geq 0$ and $\mu_1,\mu_2\in \mathcal{P}(M)$, we have
     \begin{align*}
       W_1(\mu_1P_t,\mu_2P_t)\leq c_1\e^{-\lambda t}W_1(\mu_1,\mu_2).
     \end{align*}
   \end{enumerate}
 \end{theorem}

\begin{remark}\label{rem-mainresult1}
  Since $r_0$ and $k_3$ are independent of $p$,
  the constant $\lambda$ in Theorem \ref{main-them}  also
  does not depend on $p$. Moreover, although $c_p$ depends on $p$, it can
  be controlled by a constant independent of $p$:
  \begin{align*}
    c_p&=(1+r_0)^{(p-1)/p}\exp\l(\frac1{4p}\int_0^{r_0}k_1(r)\,dr+\frac{k_2}{8p}r_0^{2+\theta}\r)\\
       &\leq (1+r_0)\exp\l(\frac14\int_0^{r_0}k_1(r)\,dr+\frac{k_2}{8}r_0^{2+\theta}\r).
  \end{align*}
\end{remark}

For the proof of Theorem \ref{main-them}, the function $\psi$ defined
below and its properties will be crucial.  First let
$\sigma\in C^1([0,\infty))$ be a function satisfying $0<\sigma\leq 1$
for $r\in (r_0,r_0+1)$, $\sigma\equiv1$ for $r\leq r_0$ and
$\sigma\equiv 0$ for $r\geq r_0+1$.  Furthermore, define
\begin{align*}
  &\ell_0(r)=4p^2r^{2(p-1)/p}\sigma(r^{1/p})^2,\\
  &\ell_1(r)=pr^{1-1/p}k_1(r^{1/p})-pk_2r^{1+\theta/p}+4p(p-1)r^{1-2/p}\sigma(r^{1/p})^2,
\end{align*}
and let
\begin{align*}
  \ell(r)=pk_2r_0^{\theta}r\1_{[0,r_0^p)}+\l(\frac{p-1}{p}\frac{\ell_0(r)}{r}-\ell_1(r)\r)\1_{[r_0^p,\infty)}.
\end{align*}
Since $\th\geq 0$, it is obvious that for $r\in (0,r_0)$,
\begin{align}\label{k1k2-ineq}
 k_1(r)-k_2r^{\theta+1}> -k_2r_0^{\theta}r.
 \end{align}
 We thus have $\ell_1+\ell>0$, according to the definitions of $\ell_1$
 and $\ell$.  Next, consider the function
 \begin{align}\label{Def-psi}
   \psi(r)=\int_0^r\exp{\l(-\int_{r_0^p}^{u}\frac{\ell_1(v)+\ell(v)}{\ell_0(v)}\,d v\r)}\,du.
 \end{align}
 The following lemma collects properties of~$\psi$.
 \begin{lemma}\label{lem1}
 Let $k_1,k_2,k_3,\theta$ and $r_0$ be given by Assumption $(A1)$.
 The function $\psi$ in \eqref{Def-psi} is well defined, twice differentiable on $(0,\infty)$,
 and satisfies $\psi'>0$ and
   $\psi''<0$. In addition,
   \begin{enumerate}[\rm(i)]
   \item for $r>0$, we have
     \begin{align*}
       \ell_1(r)\psi'(r)+\ell_0(r)\psi''(r) =-\ell(r)\psi'(r);
     \end{align*}
   \item there exist positive constants $\tilde{c}_1$ and $\tilde{c}_2$ such that
     \begin{align*}
       \tilde{c}_1r^{1/p}\leq \psi(r)\leq \tilde{c}_2r^{1/p}
     \end{align*}
     where
     \begin{align*}
       \tilde{c}_1=pr_0^{p-1} \quad\text{and}\quad
       \tilde{c}_2=pr_0^{p-1}\exp\l(\frac14\int_0^{r_0}k_1(r)\,dr+\frac{k_2}{8}r_0^{2+\theta}\r);
     \end{align*}
   \item for any $r>0$,
     \begin{align*}
       \ell (r)\psi'(r)\geq \lambda\psi(r),
     \end{align*}
     where
     $$\lambda=k_3r_0^\theta\exp\l(-\frac14\int_0^{r_0}k_1(r)\,dr-\frac{k_2}{8}r_0^{2+\theta}\r).$$
   \end{enumerate}
 \end{lemma}
 \begin{proof}
   The first assertion is immediate from the definition of $\psi$.
   For $0<r< r_0^p$, we have $\sigma(r^{1/p})=1$ and then
   \begin{align*}
     \int_{r_0^p}^r\frac{\ell(v)+\ell_1(v)}{\ell_0(v)}\,dv=\int_{r_0^p}^r\l(\frac{k_1(v^{1/p})}{4pv^{1-1/p}}+\frac{p-1}{p}v^{-1}-\frac{k_2v^{-1+\frac{2+\theta}{p}}}{4p}+\frac{k_2r_0^{\theta}v^{-1+\frac{2}{p}}}{4p}\r)\,dv.   \end{align*}
   As $k_1, k_2$ satisfy \eqref{k1k2-ineq}, we find
   \begin{align}\label{eq1(less)}
     \int_{r_0^p}^r\frac{\ell(v)+\ell_1(v)}{\ell_0(v)}\,dv&\leq
                                                            \ln r^{(p-1)/p}-\ln r_0^{p-1}
   \end{align}
   and
   \begin{align}\label{eq2(less)}
     \int_{r_0^p}^r\frac{\ell(v)+\ell_1(v)}{\ell_0(v)}\,dv&\geq \ln r^{(p-1)/p}-\ln r_0^{p-1}-\int_0^{r_0^p}\frac{k_1(v^{1/p})}{4pv^{1-1/p}}\,dv-\int_0^{r_0^p}\frac{kv^{-1+\frac{2}{p}}}{4p}\,dv.
   \end{align}
   Combining \eqref{eq1(less)} and \eqref{eq2(less)}, we conclude that
   \begin{align*}
     r_0^{p-1} r^{(1-p)/p}\leq \psi'(r)\leq \exp\l(\frac14\int_0^{r_0}k_1(r)\,dr+\frac{k}{8}r_0^{2}\r)r_0^{p-1}r^{(1-p)/p}. \end{align*}
   This implies
   \begin{align}\label{psi-est1}
     \tilde{c}_1r^{1/p}\leq \psi(r)\leq \tilde{c}_2 r^{1/p},\quad 0<r<r_0^p,
   \end{align}
   where
   \begin{align*}
     \tilde{c}_1:=p\,r_0^{p-1}\quad\text{and}\quad
     \tilde{c}_2:=p\,\exp\l(\frac14\int_0^{r_0}k_1(r)\,dr+\frac{k}{8}r_0^{2}\r)r_0^{p-1}.
   \end{align*}
   On the other hand, for $r\geq r_0^p$, we have
   \begin{align*}
     \int_{r_0^p}^r\frac{\ell(u)+\ell_1(u)}{\ell_0(u)}\, du=\frac{p-1}{p}\int_{r_0^p}^r\frac{1}{u}\, du=
     \frac{p-1}{p}(\ln r-p\ln r_0)
   \end{align*}
   which gives
   \begin{align*}
     \psi(r)&=\psi(r_0^p)+\int_{r_0^p}^{r}\exp{\l(-\int_{r_0^p}^u\frac{\ell_1(v)+\ell(v)}{\ell_0(v)}\,dv\r)}\,du\\
            &=\psi(r_0^p)+p\l(r^{1/p}r_0^{p-1}-r_0^p\r).
   \end{align*}
   Moreover,
   \begin{align*}
     \psi'(r)=r_0^{p-1}r^{(1-p)/p},\quad r\geq r_0^p.
   \end{align*}
   In particular, $\psi$ is well defined. Combining this with
   \eqref{psi-est1}, we obtain, for all $r>0$,
   \begin{align*}
     \tilde{c}_1r^{1/p}\leq \psi(r)\leq \tilde{c}_2 r^{1/p}
   \end{align*}
   where
   \begin{align*}
     \tilde{c}_1&=pr_0^{p-1},\\
     \tilde{c}_2&=p\exp\l(\frac14\int_0^{r_0}k_1(r)\,dr+\frac{k}{8}r_0^{2}\r)r_0^{p-1}.
   \end{align*}
   Using the condition $k_1(r)-k_2r^{1+\theta}\leq -k_3r^{1+\theta}$
   on $[r_0^p,\infty)$ and the above estimates for $\psi$ and $\psi'$,
   we arrive at
   \begin{align*}
     \ell(r)\psi'(r)&\geq \l(k_2 r_0^{\theta}pr\1_{[0,r_0^p)}+k_3pr^{1+\frac{\theta}{p}}\1_{[r_0^p,\infty)}\r)\psi'(r)\\
                    &\geq p r_0^{p-1}  \l(k_2 r_0^{\theta} \1_{[0,r_0^p)}+k_3r_0^{\theta}\1_{[r_0^p,\infty)}\r)r^{1/p}\\
                    &\geq \min\{k_2,k_3\}r_0^{\theta}\exp\l(-\frac14\int_0^{r_0}k_1(r)\,dr-\frac{k}{8}r_0^{2}\r)\psi(r)\\
                    &= \lambda\psi(r).\qedhere
   \end{align*}
 \end{proof}

 \begin{proof}[Proof of Theorem \ref{main-them}]
   Consider the operator $L=\Delta+Z$ where $Z$ is a vector field on
   $M$. Let $d_I$ denote the It\^{o} differential on $M$.  Then the
   $L$-diffusion process $X_t$ is obtained as solution to the
   It\^{o}-SDE
   \begin{align}\label{SDE}
     d_I X_t=\sqrt{2}\,u_t dB_t+Z(X_t)\,dt,\quad X_0=x,
   \end{align}
   where $(B_t)_{t\geq 0}$ is a $d$-dimensional standard Brownian
   motion on $\R^d$ and $(u_t)_{t\geq 0}$ a horizontal lift of
   $(X_t)_{t\geq 0}$ to the orthonormal frame bundle over $M$.
   The idea is to construct the coupling for short
   distance by reflection and for long
   distance by parallel displacement. To this end, we choose a cut-off function
   $\sigma\in C^1([0,\infty))$ as before, that is a function
   $\sigma\in C^1([0,\infty))$ satisfying $0< \sigma\leq 1$ when
   $r\in (r_0,r_0+1)$, and $\sigma\equiv 0$ when $r\geq r_0+1$ and
   $\sigma\equiv1$ when $r\leq r_0$. For $(x,y)\notin \Cut(M)$, let
   \begin{align*}
     M_{x,y}\colon\, T_xM\rightarrow T_yM,\quad
     v\mapsto P_{x,y}v-2\langle v,\dot{\gamma} \rangle(x)\dot{\gamma}(y)
   \end{align*}
   be the mirror reflection, where $\gamma$ is the minimal geodesic
   from $x$ to $y$.  We rewrite SDE \eqref{SDE} as
   \begin{align*}
     d_I X_t=\sqrt{2}\l(\sigma(\rho(X_t,Y_t)) \,u_t dB_t'+\sqrt{1-\sigma(\rho(X_t,Y_t))^2}\,u_t d B_t''\r)+Z(X_t)dt,
   \end{align*}
   where $B_t'$ and $B_t''$ are two independent Brownian motions.  Now
   define $Y_t$ as solution to the following SDE on $M$ with initial
   condition $Y_0=y$:
   \begin{align}
     d_I Y_t&=\sqrt{2}\l(\sigma(\rho(X_t, Y_t))\,M_{X_t,Y_t}u_td B_t'+\sqrt{1-\sigma(\rho(X_t,Y_t))^2}\,P_{X_t,Y_t}u_td B_t''\r)+Z(Y_t)dt.\label{coupling-eq2}
   \end{align}
   Since the coefficients of the SDE are at least $C^1$ outside the
   diagonal $\{(z,z)\colon z\in M\}$, there is a unique solution up to
   the coupling time
$$T:=\inf\{t\geq 0\colon X_t=Y_t\}.$$
As usual, we let $X_t=Y_t$ for $t\geq T$.
We ignore here technical difficulties related to a possibly non-empty cut-locus $\Cut(M)$.
It is well known how to deal with these issues, see for instance \cite[Chapt.~2]{wang2006functional} or
\cite[Sect.~3]{ATW:2006} for details.
The presence of a cut-locus actually facilitates
the coupling; it decreases the distance of the two marginal processes.

Next, we have by It\^{o}'s formula,
\begin{align*}
  d \rho(X_t,Y_t)&\leq 2\sqrt{2}\sigma(\rho(X_t,Y_t))\,d b_t+I^Z(X_t,Y_t)dt\\
                 &\leq 2\sqrt{2}\sigma(\rho(X_t,Y_t))\, d b_t
                   + \l(k_1(\rho(X_t,Y_t))-k_2\rho(X_t,Y_t)^{1+\theta}\r)\, dt,\quad t\leq T,
\end{align*}
where $b_t$ is a one-dimensional Brownian motion on $\R$.  Thus,
\begin{align*}
  d \rho(X_t,Y_t)^p&\leq p\rho(X_t,Y_t)^{p-1}\, d\rho(X_t,Y_t)
                     +\frac12p(p-1)\rho(X_t,Y_t)^{p-2}\,d \langle \rho\rangle_t \\
                   & \leq  p \rho(X_t,Y_t)^{p-1}\left\{2\sqrt{2}\sigma(\rho(X_t,Y_t))\,d b_t
                     +\left(k_1(\rho(X_t,Y_t))-k_2\rho(X_t,Y_t)^{1+\theta}\right)\,dt\right\}\\
                   & \quad + 4p(p-1)\sigma(\rho(X_t,Y_t))^2\rho(X_t,Y_t)^{p-2}\,dt,\quad t\leq T,
\end{align*}
where $\langle \rho\rangle_t$ denote the quadratic variation of  $\rho(X_t,Y_t)$.

Taking this calculation into account, our next step is to look at
properties of the process $\psi(\rho(X_t,Y_t)^p)$. First of all, by It\^{o}'s
formula, we have
\begin{align*}
  d\psi(\rho(X_t,Y_t)^p)&\leq\psi'(\rho(X_t,Y_t)^{p})\l(2\sqrt{2}p\rho(X_t,Y_t)^{p-1}\sigma(\rho(X_t,Y_t))\, d b_t
                          +\ell_1(\rho(X_t,Y_t)^p)\, dt\r)\\
                        &\quad+\psi''(\rho(X_t,Y_t)^p)\ell_0(\rho(X_t,Y_t)^p)\,dt\\
                        &= dM_t -\ell(\rho(X_t,Y_t)^p)\psi'(\rho(X_t,Y_t)^p)\, dt,\quad t\leq T,
\end{align*}
where
\begin{align*}
  d M_t=2\sqrt{2}p\psi'(\rho(X_t,Y_t)^p)\rho(X_t,Y_t)^{p-1}\sigma(\rho(X_t,Y_t))\, d b_t.
\end{align*}
By means of Lemma \ref{lem1}\,(iii), we get
\begin{align*}
  d\psi(\rho(X_t,Y_t)^p)\leq d M_t-\lambda\psi(\rho(X_t,Y_t)^p)\, dt,\quad t\leq T.
\end{align*}
Let $\tau_n=\{t\geq 0:\, \rho(X_t,Y_t)\notin [1/n, n]\}$. Then
$\tau_n\uparrow T$ as $n\rightarrow \infty$, and for $s\leq t$,
\begin{align}\label{Esti-Epsi}
  &\E\psi(\rho^p(X_{t\wedge\tau_n},Y_{t\wedge\tau_n}))\notag\\
  &\leq \E\psi(\rho^p(X_{s\wedge \tau_n},Y_{s\wedge \tau_n}))
    -\lambda\int_{s}^{t}\E\psi(\rho(X_{r\wedge \tau_n},Y_{r\wedge \tau_n})^p)\,dr.
\end{align}
From now on, for the sake of brevity, we simply write
$\rho^p_t:=\rho(X_t,Y_t)^p$.  Since $\psi(0)=0$ and $X_t=Y_t$ for
$t\geq T$, we have
\begin{align}\label{est-psi-rho}
  \E\psi(\rho_{t\wedge T}^p)=\E\l[\psi(\rho_t^p)\1_{\{t<T\}}\r]+\E\l[\psi(\rho_{T}^p)\1_{\{t\geq T\}}\r]=\E\psi(\rho_t^p).
\end{align}
Letting $n\rightarrow \infty$ of \eqref{Esti-Epsi} and using \eqref{est-psi-rho}, we conclude
that
\begin{align*}
  \E \psi(\rho_t^p)\leq \E \psi(\rho_s^p)-\lambda\int_s^t\E\psi(\rho_r^p)\, dr.
\end{align*}
Thus, letting
\begin{align*}
  f(t)=\E\psi(\rho^p_t),
\end{align*}
we obtain
\begin{align*}
  f(t)\leq f(s)-\lambda\int_s^tf(r)\, dr.
\end{align*}
For the function
\begin{align*}
  U(t)=\e^{-\lambda t}\psi(\rho(x,y)^p),
\end{align*}
it is immediate that
\begin{align*}
  U(t)=U(s)-\lambda \int_s^tU(r)\, dr, \quad U(0)=\psi(\rho(x,y)^p).
\end{align*}
This implies
\begin{align*}
  f(t)\leq U(t),\quad t\geq 0.
\end{align*}
Actually assume that there exists $t_0>0$ such that
$f(t_0)\geq U(t_0)$. Setting
$t_1=\sup\{s\leq t_0: \, f(s)\leq U(s)\}$, by the continuity of $f$
and $U$, we obtain that $f(t_1)=U(t_1)$ and $f(r)>U(r)$ for
$r\in (t_1,t_0)$. From this we conclude that
\begin{align*}
  f(t)&\leq f(t_1)-\lambda \int_{t_1}^tf(r)\, dr <U(t_1)-\lambda \int_{t_1}^tU(r)\, dr= U(t),
\end{align*}
that is
\begin{align}\label{esti-rho}
  \E\psi (\rho(X_t,Y_t)^p)\leq \e^{-\lambda t}\psi(\rho(x,y)^p).
\end{align}
Recall from Lemma \ref{Esti-Epsi}\,(ii) that there exist two constants
$\tilde{c}_1$ and $\tilde{c}_2$ such that
\begin{align}\label{psi_pinching}
  \tilde{c}_1r^{1/p}\leq \psi(r)\leq \tilde{c}_2r^{1/p}.
\end{align}
Combining \eqref{psi_pinching} with \eqref{esti-rho} we obtain the following estimate:
\begin{align}\label{est-radial-process}
  \E\rho(X_t,Y_t)&\leq \frac{1}{\tilde{c}_1}\E\psi(\rho(X_t,Y_t)^p)\leq \frac{\tilde{c}_2}{\tilde{c}_1}\e^{-\lambda t}\rho(x,y).
\end{align}
Recall that
\begin{align*}
  d\psi(\rho(X_t,Y_t)^p)\leq  dM_t -\ell(\rho(X_t,Y_t)^p)\psi'(\rho(X_t,Y_t)^p)\, dt.
\end{align*}
Since $\sigma(\rho(X_t,Y_t))=0$ for $\rho(X_t,Y_t)\geq r_0+1$ while
$d \psi(\rho(X_t,Y_t)^p)<0$ when $\rho(X_t,Y_t)\geq r_0+1$, we have
\begin{align*}
  \psi(\rho(X_t,Y_t)^p)\leq \psi\big((r_0+1)^p\vee \rho^p(x,y)\big),
\end{align*}
which together with \eqref{est-radial-process} implies
\begin{align*}
  \E^{(x,y)} \left[\rho(X_t,Y_t)^p\right]
  &\leq \big((1+r_0)\vee\rho(x,y)\big)^{p-1}\E[\rho(X_t,Y_t)]\\
  &\leq \frac{\tilde{c}_2}{\tilde{c}_1}\e^{-\lambda t}(1+r_0)^{p-1}\rho(x,y)\vee\rho(x,y)^{p}.\qedhere
\end{align*}
\end{proof}

According to Remark \ref{remark}, under the assumption that 
\begin{align*}
  \liminf_{\rho(x)\rightarrow \infty}\Ric^Z(x)>0,
\end{align*}
we can find positive constants $k_1$ and $k_2$ such that
$$I^Z(x,y)\leq k_1-k_2\rho(x,y),$$
and then by Theorem \ref{main-them}, there exist constants $c$ and $\lambda$ such
that \eqref{W-add-1} holds.  More precisely, we have now the following
results with explicit values for $c$ and $\lambda$.

 \begin{corollary}\label{cor1}
   Assume that
   \begin{align}\label{K1K2-condition}
     I^Z(x,y)\leq k_1-k_2\rho(x,y),
   \end{align}
   for some constants $k_1\geq 0$ and $k_2>0$. Then,
   \begin{enumerate}[\rm(i)]
   \item for $p>1$, $t\geq 0$, and $x,y \in M$,
     \begin{align*}
       W_{p}(\delta_xP_t, \delta_yP_t)\leq \l(1+\frac{2k_1}{k_2}\r)^{(p-1)/p} \exp\l(\frac{k_1^2}{pk_2}-\frac{k_2}{2p\e^{k_1^2/k_2}}t\r)(\rho(x,y)\vee\rho(x,y)^{1/p});
     \end{align*}
   \item for $p>1$, $t\geq 0$ and $\mu_1,\mu_2\in \mathcal{P}(M)$,
     \begin{align*}
       \tilde{W}_p(\mu_1 P_t, \mu_2P_t)\leq \l(1+\frac{2k_1}{k_2}\r)^{(p-1)/p}\exp\l(\frac{k_1^2}{pk_2}-\frac{k_2}{2p\e^{k_1^2/k_2}}t\r)\tilde{W}_p(\mu_1,\mu_2),
     \end{align*}
     where
 $$ \tilde{W}_{ p}(\mu_1,\mu_2)=\inf_{\pi\in \mathcal{C}(\mu_1,\mu_2)}\l(\int_{M\times M} \rho(x,y)^p\vee \rho(x,y) \pi(d x, d y)\r)^{1/p};$$
\item in particular, for $t\geq 0$,
  \begin{align*}
    \E^{(x,y)}\rho(X_t,Y_t)\leq  \exp\l(\frac{k_1^2}{k_2}-\frac{k_2}{2\e^{k_1^2/k_2}}t\r)\rho(x,y).
  \end{align*}
\end{enumerate}
\end{corollary}
\begin{proof}
 By assumption, we have
 $$I^Z(x,y)\leq k_1-k_2\rho(x,y).$$
 Let $r_0=2k_1/k_2$. Then, for $r\geq r_0$, we have $k_1\leq k_2 r/2$, or equivalently,
 \begin{align*}
   k_1-k_2r\leq -\frac12k_2r.
 \end{align*}
 Thus, we find $k_3=k_2/2$ and
 \begin{align*}
   \lambda&=k_3\exp{\l(-\frac14\int_0^{r_0}k_1\,dr-\frac{k_2}{8}r_0^2\r)}=\frac{k_2}{2}\exp{\l(-\frac{k_1^2}{k_2}\r)}.
 \end{align*}
 Substituting the explicit constants in the results of Theorem
 \ref{main-them}, we complete the proof.
\end{proof}

 \begin{corollary}\label{cor2}
   Keeping the assumptions as in Theorem \ref{main-them}, we have
   \begin{align*}
     |\nabla P_{t}f|\leq c_1\e^{-\lambda t}\|\nabla f\|_{\infty}^{\mathstrut}
   \end{align*}
   for any $t\geq 0$ and any $f\in C_0^{\infty}(M)$, where $c_1$ and
   $\lambda$ are the constants given in Theorem \ref{main-them}.
 \end{corollary}
 \begin{proof}
   For $f\in C_0^{\infty}(M)$, according to the definition of
   $|\nabla P_{t}f|$, we have
   \begin{align*}
     |\nabla P_{t}f|(x)
     &=\lim_{\rho(x,y)\rightarrow0}\l|\frac{P_{t}f(x)-P_{t}f(y)}{\rho(x,y)}\r|\\
     &=\lim_{\rho(x,y)\rightarrow0}\E^{(x,y)}\l[\frac{f(X_t)-f(Y_t)}{\rho(X_t,Y_t)}\,\frac{\rho(X_t,Y_t)}{\rho(x,y)}\r]\\
     &\leq c_1\e^{-\lambda t}\|\nabla f\|_{\infty}^{\mathstrut}
   \end{align*}
   for $t\geq 0$.
 \end{proof}

\section{Exponential contraction in Wasserstein distance on evolving manifolds}
In this section, we deal with the case that the underlying manifold
carries a geometric flow of complete Riemannian metrics. More precisely,
we consider a $d$-dimensional differentiable manifold $M$ equipped
with a $C^1$ family of complete Riemannian metrics
$(g_t)_{t\in (-\infty,T_c)}$ for some $T_c\in (-\infty,\infty]$.
We denote the interval $(-\infty,T_c)$ by
$I$.

We first give some quantitative results concerning exponential
contraction in Wasserstein distance over evolving manifolds. As
application, we use the $W_1$-contraction inequality to derive a
gradient inequality and uniqueness for the evolution system of
measure.

\subsection{Main results}
Let $\nabla^t$ be the Levi-Civita connection and $\Delta_t$ the
Laplace-Beltrami operator associated with the Riemannian metric $g_t$.
In addition, let $(Z_t)_{t\in [0,T_c)}$ be a $C^{1}$-family of vector
fields on $M$.  We set
\begin{align*}
  I^Z(t,x,y)=I(t,x,y)+\langle Z_t, \nabla^t\rho_t(\newdot,y) \rangle_t+\langle Z_t, \nabla^t \rho_t(\newdot,x) \rangle_t
\end{align*}
where
\begin{align*}
  I(t,x,y)=\int_0^{\rho_t(x,y)}\sum_{i=1}^{d-1}\l\{|\nabla^t_{\dot{\gamma}}J^t_i|_t^2-\langle R_t(\dot{\gamma},J^t_i)\dot{\gamma}, J^t_i \rangle_t \r\}(\gamma_s)+\partial_tg_t(\dot{\gamma},\dot{\gamma})(\gamma_s)\,ds.
\end{align*}
Now $\rho_t$ is the Riemannian distance, $R_t$ the Riemann tensor,
and $\gamma\colon[0,\rho_t(x,y)]\to M$ the minimal geodesic from $x$
to $y$ with unit speed, everything taken with respect to the
Riemannian metric $g_t$; in addition, $\{J_i^t\}_{i=1}^{d-1}$ are
Jacobi fields along $\gamma$ such that
\begin{align*}
  J^t_i(y)=P_{x,y}^tJ_i^t(x),\quad i=1,\ldots, d-1,
\end{align*}
in terms of the parallel transport
$P_{x,y}^t\colon T_xM \rightarrow T_yM$ along the geodesic $\gamma$,
and such that
$$\{\dot{\gamma}(s), J^t_i(s)\colon 1\leq i\leq d-1\},\quad s=0,\ \rho_t(x,y)$$ are
orthonormal bases of the tangent spaces $T_xM$, respectively $T_yM$,
with respect to $g_t$.

We first give a precise formulation of our assumptions in the time-dependent case.

\begin{ass2}
  {\it There exist a non-negative continuous function $k_1\in C(0,\infty)$, a positive constant $k_2$
    and a constant $\theta\geq 0$ such that
 \begin{align}\label{main-cond2}
 I^Z(t,x,y)\leq k_1(\rho_t(x,y))-k_2\rho_t(x,y)^{1+\theta}
 \end{align}
 and such that there exist positive constants $k_3$ ($k_3<k_2)$ and $r_0$ with the property:
 \begin{align*}
 k_1(r)-k_2r^{1+\theta}\leq -k_3r^{1+\theta},\quad r\geq r_0,
 \end{align*}
 and $\int_0^rk_1(v)\, dv<\infty$ for each $r>0$.}
\end{ass2}
Consider the operator $L_t=\Delta_t+Z_t$ where $Z_t$ is
   a family of vector fields which is $C^1$ in~$t$.
   Let $(X_t)$ be the diffusion process generated by $L_t$
   which is assumed to be non-explosive
   up to time $T_c$, and let
  $P_{s,t}$ be the corresponding time-inhomogeneous semigroup.

 \begin{theorem}\label{main-them2}
   Assume that Assumption $(A2)$ holds.  Then
   \begin{enumerate}[\rm(i)]
   \item for $x,y\in M$, $p\geq 1$ and $s\leq t< T_c$,
     \begin{align*}
       W_{p,t}(\delta_xP_{s,t},\delta_yP_{s,t})\leq c_p\e^{-\lambda(t-s)/p}(\rho_s(x,y)\vee\rho_s(x,y)^{1/p}),
     \end{align*}
     where
     \begin{align}
       c_p&=(1+r_0)^{(p-1)/p}\exp\l(\frac1{4p}\int_0^{r_0}k_1(r)\,dr+\frac{k_2}{8p}r_0^{2+\theta}\r),\label{cp}\\
       \lambda&=k_3r_0^\theta\exp\l(-\frac14\int_0^{r_0}k_1(r)\,dr-\frac{k_2}{8}r_0^{2+\theta}\r);\label{lambda}
     \end{align}
   \item for $s\leq t< T_c$, $p>1$ and $\mu_1,\mu_2\in \mathcal{P}(M)$, we
     have
     \begin{align*}
       \tilde{W}_{p,t}(\mu_1 P_{s,t}, \mu_2P_{s,t})\leq c_p \e^{-\lambda(t-s)/p} \tilde{W}_{p,s}(\mu_1, \mu_2),
     \end{align*}
     where
     $$\tilde{W}_{p,t}(\mu_1,\mu_2)=\inf_{\pi\in\mathcal{C}
       (\mu_1,\mu_2)}\l(\int_{M\times M} \rho_t(x,y)^p\vee \rho_t(x,y)
     \pi(d x, d y)\r)^{1/p};$$

   \item for $s\leq t< T_c$ and $\mu_1,\mu_2\in \mathcal{P}(M)$,
     \begin{align*}
       W_{1,t}(\mu_1P_{s,t},\mu_2P_{s,t})\leq c_1\e^{-\lambda (t-s)}W_{1,s}(\mu_1,\mu_2).
     \end{align*}
   \end{enumerate}
 \end{theorem}

 \begin{proof}
   Let $X_t$
   be the $L_t$-diffusion process, which we assume to be non-explosive.
   It is well known that the process $X_t$ solves the following SDE:
   \begin{align}\label{SDE-2}
     d_I X_t=\sqrt{2}u_t dB_t+Z_t(X_t)dt,\quad X_s=x,
   \end{align}
   where $(B_t)_{t\geq s}$ is a $d$-dimensional Brownian motion on
   $\R^d$.  Here $(u_t)_{t\geq s}$ is a horizontal lift of
   $(x_t)_{t\geq s}$ to the frame bundle over $M$ such that the parallel transport
   $u_tu_s^{-1}\colon (T_{x}M,g_s)\to (T_{X_t}M,g_t)$ along $x_t$ is isometric.
   We may rewrite SDE
   \eqref{SDE-2} as
   \begin{align*}
     d_I X_t=\sqrt{2}\l(\sigma(\rho_t(X_t,Y_t))\, u_t dB_t'+\sqrt{1-\sigma(\rho_t(X_t,Y_t))^2}\,u_t d B_t''\r)+Z_t(X_t)dt,
   \end{align*}
   where $B_t'$ and $B_t''$ are two independent Brownian motion on
   $\R^d$.  Recall that $\sigma\in C^1([0,\infty))$ is a function
   satisfying $0< \sigma\leq 1$ when $r\in (r_0,r_0+1)$, and
   $\sigma\equiv 0$ when $r\geq r_0+1$ and $\sigma\equiv1$ when
   $r\leq r_0$. Let $Y_t$ solve the following SDE on $M$ (with initial
   condition $Y_s=y$):
   \begin{align*}
     d_I Y_t&=\sqrt{2}\l(\sigma(\rho_t(X_t, Y_t))\,M^t_{X_t,Y_t}\,u_td B_t'+\sqrt{1-\sigma(\rho_t(X_t,Y_t))^2}\,P^t_{X_t,Y_t}u_td B_t''\r)+Z_t(Y_t)dt,
   \end{align*}
   where $P^t_{X_t,Y_t}$ and $M^t_{X_t,Y_t}$ denote respectively the parallel
   transport and the mirror reflection along the $g_t$-geodesic connecting
   $X_t$ and $Y_t$ with respect to the metric $g_t$.  Since the
   coefficients of the SDE are at least $C^1$ outside the diagonal
   $\{(z,z)\colon z\in M\}$, it has a unique solution up to the
   coupling time
$$T:=\inf\{t\geq s\colon X_t=Y_t\}.$$
Let $X_t=Y_t$ for $t\geq T$ as usual.  Then, by It\^{o}'s formula
(see \cite{Ch17}), we have
\begin{align*}
  d \rho_t(X_t,Y_t)&\leq 2\sqrt{2}\,d b_t+I^Z(t,X_t,Y_t)dt\\
                   &\leq 2\sqrt{2}\, d b_t+ (k_1(\rho_t(X_t,Y_t))-k_2\rho_t(X_t,Y_t)^{1+\theta})\, dt,\quad t\leq T,
\end{align*}
where $b_t$ is a one-dimensional Brownian motion on $\R$.
Moreover,
\begin{align*}
  d \rho_t(X_t,Y_t)^p&\leq p\rho_t(X_t,Y_t)^{p-1}\, d\rho_t(X_t,Y_t)+\frac12p(p-1)\rho_t(X_t,Y_t)^{p-2}\,d \langle \rho\rangle_t \\
                     & \leq  p \rho_t(X_t,Y_t)^{p-1}\l\{2\sqrt{2}\,d b_t+\l(k_1(\rho_t(X_t,Y_t))-k_2\rho_t(X_t,Y_t)^{1+\theta}\r)\,dt\r\}\\
                     & \quad + 4p(p-1)\rho_t(X_t,Y_t)^{p-2}\,dt.
\end{align*}
Then, by the It\^{o} formula for $\psi(\rho_t(X_t,Y_t)^p)$, we have
\begin{align*}
  d\psi(\rho_t(X_t,Y_t)^p)&\leq\psi'(\rho_t(X_t,Y_t)^{p})\l(2\sqrt{2}p\rho_t(X_t,Y_t)^{p-1}\, d b_t+\ell_1(\rho_t(X_t,Y_t)^p)\, dt\r)\\
                              &\quad+\psi''(\rho_t(X_t,Y_t)^p)\ell_0(\rho_t(X_t,Y_t)^p)\,dt\\
  &= dM_t -\ell(\rho_t(X_t,Y_t)^p)\psi'(\rho_t(X_t,Y_t)^p)\,dt
\end{align*}
where
\begin{align*}
  d M_t=2\sqrt{2}p\psi'(\rho_t(X_t,Y_t)^p)\rho_t(X_t,Y_t)^{p-1}\, d b_t.
\end{align*}
The remaining steps are the same as in the proof of Theorem
\ref{main-them}. We skip the details.
\end{proof}

 \begin{remark}
   It is natural to ask whether contraction in Wasserstein
   distance still holds when the curvature condition \eqref{main-cond2} is
   weakened as follows: there exist non-negative functions
   $k_1, k_2\in C^1(I)$ and $\phi\in C([0,\infty))$ such that
   \begin{align}\label{general-condition}
     I^Z(t,x,y)\leq k_1(t)\phi(\rho_t(x,y))-k_2(t)\rho_t(x,y)^{1+\theta}.
   \end{align}
   A possible way to deal with this case is to prove the result for
   each interval $[s,t]$. Assume that for an interval
   $[s,t]\subset I$,
   \begin{align*}
     I^Z(u,x,y)\leq k_1(s,t)\phi(\rho_u(x,y))-k_2(s,t)\rho_u(x,y)^{1+\theta},\quad
     u\in [s,t],
   \end{align*}
   and there exist $k_3(s,t)$ and $r_0(s,t)$ such that
   \begin{align*}
     & k_1(s,t)\phi(r)-k_2(s,t)r^{1+\theta}\leq -k_3(s,t)r^{1+\theta},\quad r\geq r_0(s,t)
   \end{align*}
   and $\int_0^r \phi(u)\,du<\infty$ for $r>0$. Then, by an analogous
   procedure as in the proof of Theorem~\ref{main-them2}, we get
   \begin{align*}
     W_{p,t}(\delta_xP_{s,t},\delta_yP_{s,t})
     \leq c_p(s,t) \e^{-\lambda(s,t)(t-s)/p}\big(\rho_s(x,y)\vee\rho_s(x,y)^{1/p}\big).
   \end{align*}
   Hence, if the coefficient
   $c_p(s,t) \e^{-\lambda(s,t)(t-s)/p}$ converges to $0$, as $t-s\rightarrow \infty$, we still have
   contraction of the Wasserstein distance $\tilde{W}_{p,t}$.
 \end{remark}

 Assume that ${\Ric}_t^Z\geq k(\rho_t)$ and
 $\lim\inf_{r\rightarrow \infty}k(r)>0$.  Then there exist positive
 constants $k_1$ and $k_2$ such that
 \begin{align*}
   I(t,x,y)\leq k_1-k_2 \rho_t(x,y).
 \end{align*}
 In this case, the following corollary follows directly from Theorem \ref{main-them2}.
 
 \begin{corollary}\label{main-theo-3}
   Suppose that
   \begin{align}\label{Cond:IZt}
     I^Z(t,x,y)\leq k_1-k_2\rho_t(x,y),\quad t\in I
   \end{align}
   for some non-negative constant $k_1$ and positive constant
   $k_2$. Then,
   \begin{enumerate}[\rm(i)]
   \item for $p>1$, $s\leq t<T_c$, and $x,y \in M$,
     \begin{align*}
       W_{p,t}(\delta_xP_{s,t}, \delta_yP_{s,t})&\leq \l(1+\frac{2k_1}{k_2}\r)^{(p-1)/p}\exp\l(\frac{k_1^2}{p k_2}-\frac{k_2}{2p\e^{k_1^2/k_2}}(t-s)\r)\\
       &\qquad\times(\rho_s(x,y)\vee\rho_s(x,y)^{1/p});
     \end{align*}
   \item for $s\leq t<T_c$, $p>1$ and $\mu_1, \mu_2\in \mathcal{P}(M)$,
     \begin{align*}
       \tilde{W}_{p,t}(\mu_1 P_{s,t}, \mu_2P_{s,t})\leq \l(1+\frac{2k_1}{k_2}\r)^{(p-1)/p}\exp\l(\frac{k_1^2}{pk_2}-\frac{k_2}{2p\e^{k_1^2/k_2}}(t-s)\r)\tilde{W}_{p,s}(\mu_1,\mu_2)
     \end{align*}
     where
 $$ \tilde{W}_{p,s}(\mu_1,\mu_2)=\inf_{\pi\in \mathcal{C}(\mu_1,\mu_2)}\l(\int_{M\times M} \rho_s(x,y)^p\vee \rho_s(x,y) \pi(d x, d y)\r)^{1/p};$$
\item in particular, for $s\leq t<T_c$,
  \begin{align*}
    \E^{(x,y)}\rho_t(X_t,Y_t)\leq  \exp\l(\frac{k_1^2}{k_2}-\frac{k_2}{2\e^{k_1^2/k_2}}(t-s)\r)\rho_s(x,y).
  \end{align*}
\end{enumerate}
\end{corollary}

We now apply Theorem \ref{main-them2}\,(iii) to derive gradient estimates for the 2-parameter
semigroup $P_{s,t}$.
\begin{corollary}\label{cor3}
  Under the same conditions as in Theorem \ref{main-them2}, we have
  \begin{align*}
    |\nabla^s P_{s,t}f|_s\leq c_1\e^{-\lambda(t-s)}\l\||\nabla^t f|_t\r\|_{\infty}
  \end{align*}
  for any $s\leq t $ and any $f\in C_0^{\infty}(M)$, where $c_1$ and
  $\lambda$ are defined as in \eqref{cp} and \eqref{lambda}
  respectively.
\end{corollary}
\begin{proof}
  For $f\in C_0^{\infty}(M)$, according to the definition of
  $\nabla^sP_{s,t}f$, we have for $s\le t$,
  \begin{align*}
    |\nabla^sP_{s,t}f|_s(x)&=\lim_{\rho_s(x,y)\rightarrow0}\l|\frac{P_{s,t}f(x)-P_{s,t}f(y)}{\rho_s(x,y)}\r|\\
                           &=\lim_{\rho_s(x,y)\rightarrow0}\E^{(s,(x,y))}\l(\frac{f(X_t)-f(Y_t)}{\rho_t(X_t,Y_t)}\,\frac{\rho_t(X_t,Y_t)}{\rho_s(x,y)}\r)\\
                           &\leq c_1\e^{-\lambda(t-s)}\l\||\nabla^tf|_t\r\|_{\infty}.\qedhere
  \end{align*}
\end{proof}

\subsection{Applications}

Let us first recall the notion of an evolution system of measures for
a 2-parameter semigroup.  A family of Borel probability measures
$(\mu_t)_{t\in I}$ on $M$ is called an evolution system of measures
for $P_{s,t}$ (see \cite{DaR08}) if
$$\int_M P_{s,t}\phi d\mu_s=\int \phi d\mu_t,\quad \phi\in
\mathscr{B}_b(M)$$ for $s\leq t< T_c.$ In \cite{ChT18}, we
investigated existence and uniqueness of evolution systems of
measures. The condition for uniqueness (H3) in \cite[Theorem
2.3]{ChT18} requires that the lower bound of
$ \Ric_t^Z-\frac12\partial_tg_t$ depends only on time $t$ and
satisfies an integrability condition. Here we give another condition
in terms of lower bounds on $ \Ric_t^Z-\frac12\partial_tg_t$ depending
on the radial distance $\rho_t$.

 \begin{theorem}\label{cor4}
   Suppose that there exists a function $k\in C([0,\infty))$
   with $\liminf_{s\rightarrow\infty}k(s)>0$ such that
   \begin{align}\label{RicZ-t}
     \Ric_t^Z-\frac12\partial_tg_t\geq k(\rho_t)
   \end{align}
   and that there exist
   $\varepsilon>0$ and $x_0\in M$ such that for some constant $C$,
 $$3k_{\varepsilon}(t)\varepsilon+|Z_t|_t(x_0)\leq C,\quad t\in I,$$
  where
 \begin{align}\label{Eq:k_epsilon}
   k_{\varepsilon}(t):=\sup\{\Ric_t(x)\colon\rho_t(x_0,x)\leq \varepsilon\}.
 \end{align}
 Then there exists a unique evolution system of measures
 $(\mu_s)_{s\in I}$ for $P_{s,t}$.
\end{theorem}
\begin{proof}
  First of all, by \cite[Lemma 9]{KP}, we have
  \begin{align*}
    (L_t+\partial_t)\rho_t^2&=2\rho_t(L_t+\partial_t)\rho_t+2\\
                            &\leq 2\left(F_t(\rho_t)-\int_0^{\rho_t}k(\rho_t(\gamma(s)))\,ds+|Z_t(x_0)|_t\right)\rho_t+2,
  \end{align*}
  where
  $$F_t(s)=\sqrt{k_{\varepsilon}(t)(d-1)}\coth\l(\sqrt{k_{\varepsilon}(t)/(d-1)}(s\wedge
  \varepsilon)\r)+k_{\varepsilon}(t)(s\wedge \varepsilon)$$ and
  $k_{\varepsilon}(t)$ is given by Eq.~\eqref{Eq:k_epsilon}.
As $\Ric_t^Z-\frac12\partial_tg_t\geq k(\rho_t)$
  and $\liminf_{s\rightarrow\infty}k(s)>0$, the function
  $k$ is bounded below and there exist constants $r_0>0$ and $\kappa>0$ such that for
  $r\geq r_0$,
  \begin{align*}
    k(r)\geq\kappa>0.
  \end{align*}
  By straightforward estimates, using the obvious inequality
  $\coth(x)\leq1+\frac1x$, we obtain
  \begin{align*}
    (L_t+\partial_t)\rho_t^2\leq 2d+\l(3k_{\varepsilon}(t)\epsilon+|Z_t|_t(x_0)+3(d-1)\varepsilon^{-1} \r)\rho_t
                              +c\rho_t-2\kappa \rho_t^2.
  \end{align*}
  Thus, if $3k_{\varepsilon}(t)\epsilon+|Z_t|_t(x_0)\leq C$,
  we can find constants $C_1$ and $C_2$ such that
  \begin{align*}
    (L_t+\partial_t)\rho_t^2\leq C_1-C_2\rho_t^2.
  \end{align*}
  Therefore, by \cite[Theorem 2.3]{ChT18}, there exists an
  evolution system of measures $(\mu_s)$ such that
  \begin{align*}
    \sup_{s\in (-\infty,t]} \mu_s(\rho_s^2)\leq \frac{C_1}{C_2}<\infty.
  \end{align*}
  Recall that $\Ric_t^Z-\frac12\partial_tg_t\geq k(\rho_t)$ with $k(r)>\kappa>0$ for $r_0>0$.
  Moreover, given condition \eqref{RicZ-t}, we know that there exist positive constants $k_1$ and $k_2$ such that
  \begin{align*}
    I(t,x,y)\leq - \int_0^{\rho_t(x,y)}\l(\Ric_t^Z-\frac12\partial_tg_t\r)(\dot{\gamma}(s),\dot{\gamma}(s))ds
    \leq k_1-k_2 \rho_t(x,y).
  \end{align*}
  Hence condition \eqref{Cond:IZt} in Corollary \ref{main-theo-3} is satisfied, and by
  this corollary, there exist constants $c_1$ and $\lambda$ depending
  on $k_1$ and $k_2$ such that
  \begin{align*}
    |P_{s,t}f(x_0)-\mu_t(f)|&=\l|\int (P_{s,t}f(x_0)-P_{s,t}f(y))\mu_s(d y)\r|\\
                            &=\l|\int \E^{(s,(x_0,y))}\l[\frac{f(X_t)-f(Y_t)}{\rho_t(X_t,Y_t)}\rho_t(X_t,Y_t)\r]\mu_s(dy)\r|\\
                            &\leq \big\||\nabla^t f|_t\big\|_{\infty}\int\E^{(s,(x_0,y))}[\rho_t(X_t,Y_t)]\,\mu_s(dy)\\
                            &\leq c_1\e^{-\lambda(t-s)}\big\||\nabla^tf|_t\big\|_{\infty}\,\mu_s(\rho_s)\\
                            &\leq c_1\e^{-\lambda(t-s)}\big\||\nabla^tf|_t\big\|_{\infty}\sqrt{C_1/C_2},
  \end{align*}
  which implies
  \begin{align*}
    \lim_{s\rightarrow -\infty}|P_{s,t}f(x_0)-\mu_t(f)|=0.
  \end{align*}
  If there is another evolution system of measures $\nu_t$, then
  \begin{align*}
    |\mu_t(f)-\nu_t(f)|\leq \lim_{s\rightarrow -\infty}(|P_{s,t}f(x_0)-\mu_t(f)|+|P_{s,t}f(x_0)-\nu_t(f)|)=0,
  \end{align*}
  i.e. $\mu_t\equiv \nu_t$. This finishes the proof.
\end{proof}

\begin{remark}\label{compare-conditions}
  Comparing the above conditions to \cite[Theorem
  2.3]{ChT18}, we note that the function
  $k(r)$ is only required to be positive outside a compact set.
  If $k(r)$ is not bounded below by zero, the situation is not
  covered by \cite[Theorem 2.3]{ChT18}.
\end{remark}

It is well known that evolution systems of measures play a similar
role in the inhomogeneous setting as invariant measures for
homogeneous semigroups $P_t$. Inspired by this, we take the
system $(\mu_s)_{s\in I}$ as reference measures and study the
contraction properties of the two-parameter semigroup $P_{s,t}$.

\begin{theorem}\label{supercontractive-them}
  Keeping the assumptions of Theorem \ref{cor4} and assuming that
  $\mu_s(\e^{\varepsilon \rho_s})<\infty$ for any $\varepsilon>0$, the
  semigroup $P_{s,t}$ is supercontractive.
\end{theorem}

The idea is to first establish a dimension-free Harnack inequality
under assumption \eqref{I-ineq-condition} below.

 \begin{lemma}\label{lem1a}
   Suppose that there exist constants $k_1,k_2$ such that
   \begin{align}\label{I-ineq-condition}
     I^Z(t,x,y)\leq k_1-k_2 \rho_t(x,y).
   \end{align}
   Then, for any $p>1$, the following dimension-free Harnack
   inequality holds:
   \begin{align*}
     (P_{s,t}f)^p(x)\leq P_{s,t}(f^p)(y)\exp\l(\frac{p}{4(p-1)}\l(k_1^2(t-s)+\frac{4k_1\rho_s(x,y)}{\e^{k_2(t-s)}+1}+\frac{2k_2\rho_s(x,y)^2}{\e^{2k_2(t-s)}-1}\r)\r)
   \end{align*}
   for any non-negative function $f\in \mathcal{B}_b(M)$ and
   $s\leq t<T_c$.
 \end{lemma}

 \begin{proof}
   Let $X_t$ solve the stochastic differential equation
   \begin{align*}
     d_I X_t=\sqrt{2} u_t d B_t+Z_t(X_t)dt,\quad X_s=x,
   \end{align*}
   and let $Y_t$ solve the stochastic differential equation
   \begin{align*}
     d_IY_t=\sqrt{2}P_{X_t,Y_t}^t u_t d B_t+\big(Z_t(Y_t)+\xi(t,X_t,Y_t)\big)dt,\quad Y_s=y,
   \end{align*}
   where the function $\xi\in C^1(I\times M\times M)$ will be
   specified later.  Since the coefficients of the coupled SDE are at
   least $C^1$ outside the diagonal $\{(z,z)\colon z\in M\}$, the coupled SDE has a
   unique solution up to the coupling time
$$\tau:=\inf\{t\geq s\colon X_t=Y_t\}.$$
Let $X_t=Y_t$ for $t\geq \tau$ as usual.  By It\^{o}'s formula, we
have
\begin{align}\label{coupling}
  d\rho_t(X_t,Y_t)&\leq I^Z(t,X_t,Y_t)d t-\xi_t d t \leq (k_1-k_2\rho_t(X_t,Y_t)-\xi_t) d t,\quad t\leq \tau,
\end{align}
where $\xi_t:=\xi(t,X_t,Y_t)$.
Now, for a fixed constant $T\in (s,T_c)$, let
 $$\xi_t=k_1+\frac{2k_2\e^{k_2(t-s)}\rho_s(x,y)}{\e^{2k_2(T-s)}-1},\quad t\geq s.$$
 Then
 \begin{align*}
   \int_s^{T}(k_1-\xi_t)\e^{k_2(t-s)} dt=-\rho_s(x,y),
 \end{align*}
 and
 \begin{align*}
   \rho_T(X_T,Y_T)-\rho_s(x,y)&\leq \int_s^{T}(k_1-\xi_t)\e^{k_2(t-s)} dt-\int_s^T\rho_t(X_t,Y_t)\, dt\\
                              &\leq -\rho_s(x,y)-\int_s^T\rho_t(X_t,Y_t)\, dt.
 \end{align*}
 From this, it is easy to see that $\tau\leq T$ and hence
 $X_{T}=Y_{T}$.

 Now due to Girsanov's theorem, ${Y}$ is generated by $L_t$ under the
 weighted probability measure $R\,\P$ where the density $R$ is given by
 \begin{align*}
   R=\exp{\l(\frac{1}{\sqrt{2}}\int_s^{\tau}\l\langle\xi_t\nabla^t\rho_t(X_t,\newdot)(Y_t), P_{X_t,Y_t}^t u_t d B_t\r\rangle_t-\frac14\int_s^{\tau}\xi_t^2 dt\r)}.
 \end{align*}
 Thus,
 \begin{align*}
   (P_{s,T}f(y))^{p}\leq (P_{s,T}f^{p}(x))\big(\E R^{p/(p-1)}\big)^{p-1}.
 \end{align*}
 Since $\tau\leq T$ and
 \begin{align*}
   t\mapsto N_t:=\exp{\left(\frac{p}{\sqrt{2}(p-1)}\int_s^t\left\langle \xi_r\nabla^r\rho_r(X_r,\newdot)(Y_r), P_{X_r,Y_r}^ru_r d B_r  \right\rangle_r-\frac{p^2}{4(p-1)^2}\int_s^t\xi_r^2 dr\right)}
 \end{align*}
 is a martingale, we have $\E N_{\tau}=1$ and hence,
 \begin{align*}
   \E R^{p/(p-1)}&=\E\l[N_{\tau}\exp\l(\frac{p}{4(p-1)^2}\int_s^{\tau}\xi_r^2 dr\r)\r]\\
                 &\leq \exp \l(\frac{p}{4(p-1)^2}\int_s^{T}\xi_r^2 dr\r)\\
                 &=\exp \l(\frac{p}{4(p-1)^2}\l(k_1^2(T-s)+\frac{4k_1\rho_s(x,y)}{\e^{k_2(T-s)}+1}+\frac{2k_2\rho_s(x,y)^2}{\e^{2k_2(T-s)}-1}\r)\r).
 \end{align*}
 It follows that
 \begin{align*}
   (P_{s,T}f(x))^{p}\leq (P_{s,T}f^{p}(y))\exp \l(\frac{p}{4(p-1)}\l(k_1^2(T-s)+\frac{4k_1\rho_s(x,y)}{\e^{k_2(T-s)}+1}+\frac{2k_2\rho_s(x,y)^2}{\e^{2k_2(T-s)}-1}\r)\r),
 \end{align*}
 as claimed.
\end{proof}
\begin{proof}[Proof of Theorem \ref{supercontractive-them}]
As explained in the proof of Theorem \ref{cor4}, there exist positive constants $k_1$ and $k_2$ such that
\eqref{I-ineq-condition} holds.
  Noting that $(\mu_s)$ is the evolution system of measures and
  using Lemma \ref{lem1a}, we have
  \begin{align*}
    1&=\int_M P_{s,t}|f|^{p}(y)\mu_s(dy)\\
     &\geq |P_{s,t}f|^{p}(x)\int_M \exp \l(-\frac{p}{4(p-1)}\l(k_1^2(t-s)+\frac{4k_1\rho_s(x,y)}{\e^{k_2(t-s)}+1}+\frac{2k_2\rho_s(x,y)^2}{\e^{2k_2(t-s)}-1}\r)\r)\mu_s(dy)
    \\
     &\geq |P_{s,t}f|^{p}(x)\int_{B_s(x_0,1)} \exp \l(-\frac{p}{4(p-1)}\l(k_1^2(t-s)+\frac{4k_1(\rho_s(x)+1)} {\e^{k_2(t-s)}+1}+\frac{2k_2(\rho_s(x)+1)^2}{\e^{2k_2(t-s)}-1}\r)\r)\mu_s(dy)
    \\
     &\geq |P_{s,t}f|^{p}(x)\mu_s(B_s(x_0,1))\exp{\l(-pC(t-s,p,k_1,k_2)-\frac{p(2k_1\e^{k_2 (t-s)}+k_2-2k_1)}{(p-1) (\e^{2k_2(t-s)}-1)}\rho_s(x)^2\r)},
  \end{align*}
  where $B_s(x_0,1)=\{x\in M\colon\rho_s(x)\leq 1\}$ is the unit
  geodesic ball (with respect to $g_s$) centered at $x_0$ and $C(t-s,p,k_1,k_2)$ is a
  constant depending on $t-s$, $p$, $k_1$ and $k_2$. Letting
$$\lambda=\frac{2k_1\e^{k_2 (t-s)}+k_2-2k_1}{(p-1) (\e^{2k_2(t-s)}-1)},$$
we get
\begin{align*}
  |P_{s,t}f|(x)\leq \frac{\exp{\left(C(t-s,p,k_1,k_2)\right)}}{\mu_s\left(B_s(x_0,1)\right)^{1/p}}\e^{\lambda \rho_s^2}<\infty,\quad \mu_t(|f|^{p})=1.
\end{align*}
Therefore
$$\mu_s(|P_{s,t}f|^q)^{1/q}\leq \frac{\exp{\left(C(t-s,p,k_1,k_2)\right)}}{\mu_s(B_s(x_0,1))^{1/p}}(\mu_s(\e^{\lambda q \rho_s^2}))^{1/q}.$$
Thus if $\mu_s(\e^{\lambda q \rho_s^2})<\infty$ for some $s\in I$, then
$P_{s,t}$ is supercontractive, i.e.,
$\|P_{s,t}\|_{(p,t)\rightarrow(q,s)}<\infty$ for any $1<p<q<\infty$.
\end{proof}

\proof[Acknowledgements]
This work has been supported by Fonds National
de la Recherche Luxembourg (Open project O14/7628746 GEOMREV). Shao-Qin Zhang has been supported by the National Natural Science Foundation of China (Grant No. 11901604).

\bibliographystyle{amsplain}%

\bibliography{Bib}

\end{document}